\documentclass[psamsfonts]{amsart} 
\usepackage{amssymb}
\usepackage{bbm}

\usepackage[dvips]{graphicx}

\usepackage{graphicx}
\usepackage{amssymb}                       
\usepackage{amsmath}
\usepackage{color}
\usepackage{times}

\definecolor{mahogany}{cmyk}{0, 0.77, 0.87, 0}
\definecolor{salmon}{cmyk}{0, 0.53, 0.38, 0}
\definecolor{melon}{cmyk}{0, 0.46, 0.50, 0}
\definecolor{yellowgreen}{cmyk}{0.44, 0, 0.74, 0}
\definecolor{brickred}{cmyk}{0, 0.89, 0.94, 0.28}
\definecolor{OliveGreen}{cmyk}{0.64, 0, 0.95, 0.40}
\definecolor{RawSienna}{cmyk}{0, 0.72, 1.0, 0.45}
\definecolor{ZurichRed}{rgb}{1, 0, 0} 

\usepackage{fancyhdr}
\usepackage{amsmath,amstext,amssymb,amsopn,amsthm}
\usepackage{amsmath,amssymb,amsthm}
\usepackage[mathscr]{eucal}

\usepackage[unicode,bookmarks,colorlinks]{hyperref}
\hypersetup{
    linkcolor=brickred,
}
\begin{document}

\numberwithin{equation}{section}

\newtheorem{thm}{Theorem}[section]        
\newtheorem{cor}{Corollary}[section]
\newtheorem{lem}{Lemma}[section]
\newtheorem{prop}{Proposition}[section]
\newtheorem{rmk}{Remark}[section]
\newtheorem{example}{Example}[section]

\newcommand{\R}{\mathbb{R}}                  
\newcommand{\set}[1]{ \left\{#1\right\} }
\newcommand{\mysum}[3]{\sum\limits_{#1=#2}^{#3}}          
\newcommand{\myprod}[3]{\prod\limits_{#1=#2}^{#3}}
\newcommand{\E}[1]{\mathbb{E}\left[#1\right]}
\newcommand{\wh}[1]{\widehat{#1}}              
\newcommand{\eid}{\,\,{\buildrel \mathcal{D} \over =}\,\,}   
\newcommand{\abs}[1]{\left|#1\right|}
\newcommand{\bks}[1]{\left[#1\right]}
\newcommand{\tgo}[2]{#1\rightarrow #2}
\newcommand{\tred}[1]{\textcolor{red}{#1}}
\newcommand{\pthesis}[1]{\left(#1\right)}

\title[Additive Functionals of L\'evy processes]{a decomposition for additive functionals of 
L\'{e}vy Processes}

\author{Luis Acu\~na Valverde*}\thanks{* Supported in part  by NSF Grant
\# 0603701-DMS under PI Rodrigo Ba\~nuelos}
\address{Department of Mathematics, Purdue University, West Lafayette, IN 47907, USA}
\email{lacunava@math.purdue.edu}

\begin{abstract}
Motivated by the recent results of  Nualart and Xu \cite{Nualart} concerning limits laws for occupation times of one dimensional symmetric stable processes, this paper  proves   a decomposition for functionals of one dimensional symmetric  L\'{e}vy processes under certain conditions on  the characteristic exponent and computes the moments of the decomposition. 
\end{abstract}
\maketitle

\section{introduction}
Let $X=\set{X_t}_{t\geq 0}$ be a one-dimensional symmetric L\'{e}vy process started at zero on the probability space $\left(\Omega, \mathbb{P}, \set{\mathcal{F}_t}_{t\geq0}\right)$ with characteristic function  given by
\begin{align}\label{char}
\E{e^{-\dot{\iota}\,x\cdot X_s}}=e^{-s\,\Psi(x)}.
\end{align}

In this paper, we are interested in finding suitable positive and increasing sequences $\set{a(n): n\in\mathbb{N}}$ and $\set{b(n,t): n\in\mathbb{N}}$, $t>0$, both tending to $\infty$ as $\tgo{n}{\infty}$,   such that under appropriate  conditions on the characteristic exponent $\Psi(x)$ and the function $f$, the additive functional
\begin{equation}\label{addfunc}
\frac{1}{a(n)}\int_{0}^{b(n,t)}ds\,f(X_s) 
\end{equation}
can be decomposed as a sum of two processes 
\begin{equation}\label{decomp}
I_n^{(1)}(t)+I_n^{(2)}(t),
\end{equation}
 where $I_n^{(1)}(t)$ converges to zero in $L^p$, for some   $p\geq1$, and  $\set{I_n^{(2)}(t), n\in \mathbb{N}}$ is a uniformly integrable sequence with finite moments and with further probabilistic properties.
 
The  foregoing decomposition is of great interest since employing certain  techniques it is possible to prove weak convergence of \eqref{addfunc} to a non--degenerate random variable.  One of these techniques, that may immediately provide weak convergence, consists of  proving the existence of local times $\set{L_t(x), t\geq0, x\in \R}$ for the process $X$. This  is equivalent, according to \cite{Rosen}, to showing that
\begin{equation}\label{Elt}
\int_{\R}dx\,\Re\left(\frac{1}{1+\Psi(x)}\right)<\infty.
\end{equation} 
The local time describes the amount of time  spent by the process at $x$ in the interval $[0,t]$ and it is defined (see  \cite{Apple})  as 
\begin{align*}
L_t(x)=\varlimsup_{\tgo{\epsilon}{0+}}\frac{1}{2\epsilon}
\int_{0}^{t}ds\,\mathbbm{1}_{\set{{\abs{X_s-x}<\epsilon}}}.
\end{align*} 
We also have the occupation density formula
\begin{align*}
\int_{0}^{t}ds\,f(X_s)=\int_{\R}dx\,f(x)\,L_t(x).
\end{align*}
Particular examples of processes having local times are the symmetric $\alpha$-stable processes with $\Psi(x)=|x|^{\alpha}$, $1<\alpha\leq 2$,
 for which it can be shown by appealing to the well-known scaling property, $X_{\eta t}\overset{\mathcal{L}}=\eta^{1/\alpha}X_t$, $\eta>0$, that
\begin{equation*}
n^{\frac{1-\alpha}{\alpha}}\int_{0}^{nt}ds\,f(X_s) 
\overset{{\mathcal{L}}}{\longrightarrow} L_t(0)\int_{\R}dx\,f(x),
\end{equation*}
as $\tgo{n}{\infty}$, for $f\in L^{\infty}(\R)\cap L^1(\R)$;  see \cite{Nualart} for further details. 
Here, $\overset{\mathcal{L}}=$ means equality in law. We also refer the reader to \cite{Rosen1} for a   probabilistic approach involving local times. 

A further approach used to prove weak convergence  is the method of moments. The reader interested in a comprehensive presentation of this topic should consult \cite{Krish}.  This technique  is very restrictive since it requires the existence and the finiteness of the quantities $$m_k(t)=\lim\limits_{\tgo{n}{\infty}}\E{\left(I_{n}^{(2)}(t)\right)^k}.$$ Moreover, $\set{m_k(t):k\in \mathbb{N}, t>0}$ must  uniquely determine the distribution of a random variable. In this direction, Carleman's condition stated in \cite{Krish} asserts that 
\begin{equation}\label{CC}
\mysum{k}{1}{\infty}\pthesis{m_{2k}(t)}^{-\frac{1}{2k}}=\infty
\end{equation}
is sufficient to guarantee uniqueness. The moment  techniques have been used in the  recent paper of Nualart and Xu \cite{Nualart} where a decomposition similar to \eqref{decomp} is proved  for the symmetric Cauchy process $X$ where 
$\Psi(x)=|x|$.   There, it is also established that
\begin{equation*}
\frac{1}{n}\int_{0}^{e^{nt}}ds\,f(X_s)\overset{\mathcal{L}}{\longrightarrow}Z(t)\int_{\R}dx\,f(x),
\end{equation*}
as $\tgo{n}{\infty}$, for  all bounded functions $f$ with  
 $\int_{\R}dx\,|x|\,|f(x)|<\infty$, where  $Z(t)$ is an exponential random variable with parameter $t^{-1}$.
    
Our results in this paper are motivated by the Nualart--Xu result \cite{Nualart} and the Fourier transform techniques used by Ba\~nuelos and S\'a Barreto in \cite{Ba.Sab} and in the author's paper \cite{Acuna} to compute the heat invariants for Schr\"odinger operators.  
In order to state our main theorems, we will  impose some
conditions  not only on  the function $f$ to be considered in $\eqref{addfunc}$ but also on the characteristic exponent $\Psi(x)$. As we shall  see later, such conditions  influence  the behaviour of the transition densities for the process $X,$ denoted throughout the paper by $p_t(x).$ To begin with, we will assume that  $\Psi(x)=\Psi(|x|)\geq 0$ is a  non-decreasing function on $[0,\infty)$. 
In addition, we impose the following  (crucial) assumptions: There exists $\ell \in (0,\infty)$ such that
\begin{align}
\lim\limits_{x\rightarrow 0+}\frac{ \Psi(x)}{x^2}&=\ell,\label{ellandphi}\\
\lim\limits_{x\rightarrow \infty}\frac{ \Psi(x)}{\ln(\abs{x}+1)}&=\infty,\label{lim2}\\ 
\int_{\R}\frac{dx}{1+\Psi(x)}&<\infty. \label{Elt}
\end{align}
Thus, as we have previously mentioned, condition \eqref{Elt} implies the existence of  a local time associated with  the underlying process $X$. However, instead of appealing to this probabilistic approach,  we rather turn to the analytic point of view because
of the familiarity of the techniques in analysis. It also helps  us to obtain and handle expressions    for the higher moments of certain random variables involving only   the characteristic exponent and the transition densities of the process $X$.

We now proceed to discuss the main implications of the latter conditions imposed on $\Psi(x)$. We start by observing that
condition \eqref{ellandphi} together with the fact that 
$\Psi(x)\in \R$ implies, by means of the 
L\'evy--Khintchine formula, that
\begin{align}\label{PsiLevym}
\Psi(x)=c\,x^2+\int_{\R-\set{0}}\Pi(ds)\,(1-\cos(xs)),
\end{align}
where $\Pi$ denotes the L\'evy measure of the process $X$ and for some $c\in \R$ (see \cite{Ber} for details). Furthermore, the existence of the  limit given in  \eqref{ellandphi}  is equivalent to the finiteness of the second moment.  In fact,
$$\ell=c+\frac{1}{2}\int_{\R-\set{0}}\, \Pi(ds)\,s^2.$$ 
The aforementioned observation is easily derived from  the representation \eqref{PsiLevym} 
and the following fact, whose  proof may be found in  \cite[p.~132]{Apple}. $\E{X_t^2}<\infty$, for all $t>0$, if and only if 
\begin{align*}\label{Smc}
\int_{\abs{s}\geq 1}\Pi(ds)\,s^2<\infty.
\end{align*} 

It is convenient at this point to introduce some notations and implications  concerning the limit \eqref{ellandphi} which will be needed for the crucial proof of our main result. Let us define, for every $\delta >0,$
\begin{align}
\overline{\ell}(\delta)=\sup \limits_{\abs{x}\leq \delta}\frac{ \Psi(x)}{x^2} \,\,\,\,\,\,\,  \mbox{and} \,\,\,\,\,\,\, 
\underline{\ell}(\delta)=\inf\limits_{\abs{x}\leq \delta}\frac{ \Psi(x)}{x^2}.
\end{align}
It follows trivially from this  definition that 
\begin{align}\label{ine1}
\underline{\ell}(\delta)x^2\leq \Psi(x)\leq \overline{\ell}(\delta)x^2,\,\,\, \abs{x}\leq \delta,
\end{align}
and due to the assumption  \eqref{ellandphi}, we also have
\begin{align}\label{limitc}
\lim\limits_{\tgo{\delta}{0+}}\overline{\ell}(\delta)=\lim\limits_{\tgo{\delta}{0+}}
\underline{\ell}(\delta)=\ell.
\end{align}

Before examining the  remaining assumptions, we provide, in a general context,  some examples of symmetric L\'evy processes satisfying condition \eqref{ellandphi} which are known in the literature as Subordinated Brownian Motions. In the following, $a\wedge b$ will stand for $\min\set{a,b}$. 

\begin{example}\label{e1}
Subordinated Brownian Motions are defined, according to 
\cite{Bog}, as L\'evy processes whose characteristic exponent $\Psi(x)$ can be expressed as 
$$\Psi(x)=\phi(x^2),$$ 
for some Bernstein function $\phi(x)$. That is, 
$\phi:(0,\infty)\rightarrow [0,\infty)$ is a $C^{\infty}$--function that admits  the 
following representation  
\begin{equation}\label{intrep}
\phi(x)=\int_{0}^{\infty}\mu(ds)\left(1-e^{-xs}\right),
\end{equation}
where $\mu$ is a $\sigma$-finite measure on  $(0,\infty)$
satisfying
$\int_{0}^{\infty}\mu(ds)(s\wedge1)<\infty.$ 
In addition, in \cite{Bog} is also shown  that 
 \begin{equation*}
 \lim\limits_{\lambda\rightarrow 0^+}\frac{\phi(\lambda)}{\lambda}= \int_{0}^{\infty}\mu(ds)\,s.
 \end{equation*}
Then, as a result of the last limit, we see that \eqref{ellandphi} holds with
$\ell=\int_{0}^{\infty}\mu(ds)\,s,$ provided the integral is positive and finite.
\end{example}

We next elaborate on conditions \eqref{lim2} and \eqref{Elt} and how these influence the behaviour of the function $p_t(0)$. Based on Theorem 1 in \cite{Schilling}, the assumption \eqref{lim2} which is known as the Hartman--Wintner condition  guarantees  the existence
of the transition densities $p_t(x)$ and $p_t\in L^{1}(\R)\cap C^{\infty}(\R)$ for all $t>0$. Consequently, as
an application of the Fourier inversion formula we obtain
\begin{equation*}
p_t(x)=(2\pi)^{-1}\int_{\R}d\xi\,e^{\dot{\iota}x\cdot\xi}e^{-t\Psi(\xi)},\,\, x\in \R.
\end{equation*}
From our  assumptions on $\Psi$, we deduce from the above expression that  $p_t(x)$ is a non-negative  radial function and 
\begin{equation}\label{invdens}
p_t(0)=(2\pi)^{-1}\int_{\R}d\xi\,e^{-t\Psi(\xi)}
\end{equation} is a decreasing function of $t$.

Next, the condition \eqref{Elt}  has been  introduced to ensure that for all  $\beta>0$,
\begin{align}\label{fazero}
\int_{0}^{\beta}ds\,p_s(0)<\infty.
\end{align}
The latter fact arises from the following inequality.
\begin{align*}
\int_{\R}\frac{dx}{1+\Psi(x)}=\int_{\R}dx\int_{0}^{\infty}
ds\,e^{-s\pthesis{1+\Psi(x)}}
&\geq \int_{\R}dx\int_{0}^{\beta}
ds\,e^{-s\pthesis{1+\Psi(x)}} \\
&\geq 2\,\pi\,e^{-\beta}\int_{0}^{\beta}ds\, p_{s}(0).
\end{align*}
The last term in the above inequality is obtained by changing the order of integration and using identity \eqref{invdens}.
We also note that for every real number $a$ satisfying $a>\beta$, we have
\begin{align}\label{Ldct}
\mathbbm{1}_{\bks{\beta/a,1}}(s)\cdot p_{as}(0)\leq \mathbbm{1}_{\bks{0,1}}(s)\cdot p_{\beta}(0) \in L^{1}(\R^+),
\end{align}
where in the last expression we have used the fact that 
$p_t(0)$ is a decreasing function.
Thus, because of the identity \eqref{invdens}, we realize that $p_{as}(0)
\rightarrow 0$ as  $a\rightarrow \infty$ so that by combining \eqref{Ldct}, which allows us to apply the Lebesgue Dominated Convergence Theorem and \eqref{fazero}, we conclude that 
\begin{align}\label{asympdensity2}
 \lim\limits_{\tgo{a}{\infty}}\frac{1}{a}\,\int_{0}^{a}ds_1\,p_{s_1}(0)&=
  \lim\limits_{\tgo{a}{\infty}}\frac{1}{a}\int_{0}^{\beta}
  ds_1\,p_{s_1}(0)+\frac{1}{a}\int_{\beta}^{a}ds_1\, p_{s_1}(0) \\
 &=\lim\limits_{\tgo{a}{\infty}}\int_{\beta/a}^{1}ds\,p_{as}(0)=0.  \nonumber
\end{align}

The most relevant L\'evy processes that share all the properties previously mentioned  are the relativistic $\alpha$--stable processes with $1<\alpha<2$ to be introduced in the example below. These processes belong to the class of Subordinated Brownian Motions and play an important role in physics and Schr\"{o}dinger Operator Theory (see \cite{Ryznar, Ryznar2, Ba.Sel, Ba.Sab}).
 
 \begin{example}\label{relat}
 Let $m>0$ and $1<\alpha<2$. The one-dimensional symmetric L\'{e}vy process  with characteristic exponent given by
\begin{align*}
\Psi(x)=(x^2+m^{2/\alpha})^{\alpha/2}-m
\end{align*}
will be denoted by $X^{m,\alpha}=\set{X^{m,\alpha}_t}_{t\geq 0}$. This process is called the  relativistic $\alpha$--stable process of index $m$. It is a well known fact that $\phi(x)=(x+m^{2/\alpha})^{\alpha/2}-m$ is 
a Bernstein function (see \cite{Bog}) with measure $\mu$ as defined in 
example \ref{e1} given by
$$\mu(ds)=\frac{\alpha/2}{\Gamma(1-\alpha/2)}e^{-m^{2/\alpha}\,s}
s^{-1-\alpha/2}ds.$$  
It is a routine exercise to verify that the conditions \eqref{ellandphi} with $\ell= 2^{-1}\,\alpha\, m^{(\alpha-2)/\alpha}$, \eqref{lim2} and \eqref{Elt} hold.
\end{example}

As for the  function $f$, we shall impose  the following conditions.
\begin{enumerate}
\item[$(C0)$]$f\in L^{1}(\R)\cap L^{\infty}(\R)$. 
\item[$(C1)$]$ \wh{f}\in L^{1}(\R)$.
\item[$(C2)$]
\begin{equation}\label{deltacond}
\int_{\R}dx\,\frac{|\wh{f}(x)-\wh{f}(0)|}{x^2}<\infty.
\end{equation}
\end{enumerate}
We remark that condition $(C1)$ immediately implies $$\int_{\abs{x}>M}dx\,\frac{|\wh{f}(x)-\wh{f}(0)|}{x^2}<\infty,$$
for all $M>0$.   Therefore, $(C2)$ will hold as long as the function $x^{-2}|\wh{f}(x)-\wh{f}(0)|$ is integrable in a neighbourhood of zero. 

We denote by $\mathcal{D}$ the set of functions satisfying $(C0)$-$(C2)$. Next, we  exhibit some examples of  functions that belong to $\mathcal{D}$. 

\begin{example} 
A typical element of $\mathcal{D}$ is the Gaussian kernel $f_{1}(x)=\frac{1}{\sqrt{2\pi}}e^{-x^2/2}$ since clearly $\wh{f_{1}}(\xi)=e^{-|\xi|^2/2}$ satisfies $(C1)$.  Regarding  $(C2)$, it suffices to observe that for all $\xi \in \R$
\begin{equation*}
|e^{-|\xi|^2/2}-1|=\abs{ \int_{0}^{|\xi|^2/2}e^{-u}du}\leq \xi^2/2.
\end{equation*}

Another element of $\mathcal{D}$ is the Jackson--Vall\'{e}e--Poussin Kernel
\begin{equation*}
f_2(x)=\frac{12}{\pi}\left(\frac{\sin(x/2)}{x}\right)^4,
\end{equation*}
for which is known (\cite{Akhiezer}, p. 23) that
\begin{equation*}
\wh{f_2}(\xi)= \left\{ \begin{array}{ccc}
        1-\frac{3\xi^2}{2}+\frac{3|\xi|^3}{4} & \mbox{for $|\xi|\leq1$}, \\ \\
        \frac{1}{4}(2-|\xi|)^3 & \mbox{for $1\leq |\xi|\leq 2$},\\ \\
        0 & \mbox{for $|\xi|\geq 2$}.
        \end{array}
        \right.
\end{equation*} 
Clearly, $f_i*f_j$  with $i,j\in\set{1,2}$ are  members of $\mathcal{D}$ as well.
\end{example}

The main results of the paper are the following:
\begin{thm} \label{mainthm}
Let $\set{a(n): n\in \mathbb{N}}$ be any increasing sequence of positive numbers with $\tgo{a(n)}{\infty}$ as $\tgo{n}{\infty}$.
Fix $f\in \mathcal{D}$ and consider $\delta, t>0$.  Set 
\begin{equation*}
I_{n}(t)=\frac{1}{a(n)}\int_{0}^{t^2a^2(n)}ds\,f(X_s). 
\end{equation*}
Then, 
the process $I_{n}(t)$ can be written  as
\begin{equation*}
I^{(1)}_{n,\delta}(t)+\wh{f}(0)I^{(2)}_{n,\delta}(t),
\end{equation*}
where
\begin{enumerate}
\item[$(i)$]
$\tgo{I^{(1)}_{n,\delta}(t)}{0}\,\, \text{in $L^{2}$ as $\tgo{n}{\infty}$}$
and 
\\
\item[$(ii)$] for all $k\in \mathbb{N},$ 
\begin{align}\label{finmom2}
\sup\limits_{n\in \mathbb{N}}\E{\abs{I_{n,\delta}^{(2)}(t)}^{2k}}&\leq \pthesis{\frac{t^{2}}{4\,\underline{\ell}(2k\delta)}}^k  \frac{(2k)!}{k!}.
\end{align}
In particular, it follows that the sequence $\set{\pthesis{I^{(2)}_{n,\delta}(t)}^{k}: n\in  \mathbb{N}}$ is uniformly integrable for all $k$ and $\delta>0$. Furthermore, 
 we have    
\begin{align} 
\lim \limits_{ \tgo{\delta}{0^+}} \varliminf\limits_{ 
\tgo{n}{\infty}}\E{\left(I^{(2)}_{n,\delta}(t)\right)^k}&=
\lim\limits_{ \tgo{\delta}{0^+} }\varlimsup\limits_{\tgo{n}{\infty}}\E{\left(I^{(2)}_{n,\delta}(t)\right)^k}\label{limits}\\ 
&=\pthesis{\frac{t}{2\sqrt{\ell}}}^k
\frac{k!}{\Gamma\pthesis{1+\frac{k}{2}}}. \nonumber
\end{align}
Here, $\ell$ and $\underline{\ell}$ are the constant and the function defined in \eqref{ellandphi} and \eqref{ine1}, respectively.
\end{enumerate}
\end{thm}

We remark that the goal of  Theorem $\ref{mainthm}$  is to identify the weak limit of the   random variable $I_n(t)$. Hence, when $\wh{f}(0)=0$,  $I_n(t)$ converges weakly to zero by recalling that convergence in $L^{p}(\mathbbm{P})$, $p\geq1$ implies weak convergence.  On the other hand, the most interesting case takes place when $\wh{f}(0)\neq 0$. 
 In the following theorem, we  identity the weak limit of $I_{n}(t)$  by imposing a supplementary condition on the sequence $\set{a(n):n\in \mathbb{N}}$.

\begin{thm}\label{weaklim}
Consider all the assumptions of Theorem \ref{mainthm} with $\wh{f}(0)\neq 0$ and  the following additional condition on the sequence  $\set{a(n):n\in \mathbbm{N}}:$ 
\begin{align}\label{seqcond}
\lim\limits_{\tgo{n}{\infty}}\frac{a(n)}{a(n+N)}=1,
\end{align}
for all $N\in \mathbbm{N}$. Then, there exists a random variable
$I(t)$ such that
\begin{enumerate}
\item[$(i)$] $I_{n}(t) \overset{\mathcal{L}}{\longrightarrow}\wh{f}(0)\,I(t),$ as $\tgo{n}{\infty}$
and 
\\
\item[$(ii)$]
for all  $k\in \mathbb{N}$,  $I(t)\in L^{k}$  and
$$\E{I^k(t)}=
\pthesis{\frac{t}{2\sqrt{\ell}}}^k\,\frac{k!}{\Gamma\pthesis{1+\frac{k}{2}}}.$$
In particular, it follows that $I(t)$ is uniquely determined by its moments since  Carleman's condition \eqref{CC} holds with the aid of the Stirling's formula.
\end{enumerate}
 \end{thm}

\begin{cor}\label{maincor}
 Under the assumptions of Theorem \ref{mainthm}, we have
\begin{align}\label{firstlim}
\,\,\,\,\,\,\,\,\,\,\,\,\,\, \lim\limits_{\tgo{n}{\infty}}\E{\frac{1}{a(n)}\int_{0}^{t^2a^2(n)}ds\,f(X_s)}&=\left(\int_{\R}dx\,f(x)\right)
\frac{t}{\sqrt{\pi \ell}}, 
\end{align}
\begin{align}\label{seclim}
\lim\limits_{\tgo{n}{\infty}}\E{\frac{1}
{a^2(n)}\left(\int_{0}^{t^2a^2(n)}ds\,f(X_s)\right)^2}&=
\left(\int_{\R}dx\,f(x)
\right)^2\, \frac{t^2}{2\, \ell}.
\end{align}
\end{cor}

The paper is organized as follows. In \S \ref{sec:prelim},  we give some notations and  state a basic identity (Lemma \ref{polarcoor}) to be used  later on.  In \S\ref{sec:props}, we give  the 
proof of  Theorem \ref{mainthm} as a consequence of a series of propositions. Finally, the corresponding proofs of Theorem \ref{weaklim} and Corollary \ref{maincor} are  given in \S \ref{sec:maincor}. 

\section{Notation and preliminaries}\label{sec:prelim}
Let $k\geq 1$ be an integer. We will use $dx^{(k)}$ to denote integration with respect to the Lebesgue measure in $\R^k$ and $dx_i$ to denote integration in $\R$ so that 
 $$dx^{(k)} =dx_1dx_2...dx_k.$$

Also for any $L>0$, we define
$$D_k(L)=\set{(s_1,s_2,...,s_k)\in [0,L]^k: s_0=0< s_1< s_2<...< s_k< L}.$$

Henceforth, we will employ the following basic identity (see \cite{Tiber}) which holds for every integrable function $V$ over $[0,L]$. 
\begin{equation}\label{symmetric}
\left(\int_{0}^{L}ds\,V(s)\right)^k= k!\,\int_{D_k(L)}
ds^{(k)}\,\myprod{i}{1}{k}V(s_i). 
\end{equation}
We also observe that for $0 = s_0 < s_1 <\dots  <s_k < \infty$ and $x_1 , . . . , x_k \in \R$, we have
\begin{align}\label{incmts}
\E{\exp\left(\dot{\iota} \mysum{j}{1}{k}x_j\,X_{s_j}\right)}=\exp\left(-\mysum{i}{1}{k}\Psi\left(\mysum{j}{i}{k}x_j\right)(s_i-s_{i-1})\right).
\end{align}
The last displayed equality is a consequence of the independence of increments of the process $X$.
 
  In what  follows, $\mathbb{S}^{k-1}$ will be used to denote the sphere in $\R^k$ with surface area  
  $$|\mathbb{S}^{k-1}|=\frac{2\,\pi^{k/2}}{\Gamma(k/2)}.$$
   In addition, we set 
\begin{align*}
\mathbb{S}^{k-1}_+&=\mathbb{S}^{k-1}
\cap\set{(x_1,...,x_k):x_i>0 \,\, \mbox{for all} \,\, i\in\set{1,...k}}. 
\end{align*}

Consider $\mathcal{H}^{k-1}$  the Hausdorff measure in $\R^{k-1}$ and define for a bounded and non-decreasing function $F:[0,\infty) \rightarrow [0,\infty) $,
\begin{equation}\label{Hdefinition}
H_{F}^{(k)}(r,\xi)=\left\{
\begin{array}{cc}
F(r\,\xi) & \mbox{if \,$k=1$}, \\
\int_{\mathbb{S}^{k-1}_+}\,\mathcal{H}^{k-1}(dz^{(k)}) \,\myprod{i}{1}{k}\,F(\xi \,r\, z_i)& 
\mbox{if \,$k\geq 2$},
\end{array}
\right.
\end{equation} 
for  any $r,\xi\geq0.$
This function  has the following properties: 
\begin{enumerate}
\item[$(i)$] Fix $\xi\geq 0$. Then, $H_{F}^{(k)}(r,\xi)$ is a non--decreasing function with respect to the variable $r$ since $F$ is non-decreasing by assumption.
\item[$(ii)$]If $F(\lambda)=1$ for all $\lambda\geq 0$, then
\begin{align*}
H_F^{(k)}(r,\xi)=
\left\{
\begin{array}{cc}
1 & \mbox{if \,$k=1$}, \\
|\mathbb{S}^{k-1}_+|=2^{-k}|\mathbb{S}^{k-1}| & 
\mbox{if \,$k\geq 2$}.
\end{array}
\right.
\end{align*}
\item[$(iii)$] An application of either the Monotone Convergence Theorem ($F$ non--decreasing) or the Lebesgue Dominated Convergence Theorem ($F$ bounded and (ii) above) shows that for any
sequence $\set{r_{n}:n\in \mathbb{N}}$ with $\tgo{r_n}{\infty}$ as $\tgo{n}{\infty}$, 
\begin{equation}\label{limit}
\lim\limits_{\tgo{n}{\infty}}H_F^{(k)}(r_n,\xi)=
\left\{
\begin{array}{cc}
||F||_{\infty} & \mbox{if \,$k=1$}, \\
\\
2^{-k}\,||F||^k_{\infty}\,|\mathbb{S}^{k-1}| & 
\mbox{if \,$k\geq 2$},
\end{array}
\right.
\end{equation}
\end{enumerate}
for $\xi\geq0$ fixed (notice that when $\xi=0$, we have to replace $||F||_{\infty}$ with $F(0)$).
This shows  that 
$\lim\limits_{\tgo{r}{\infty}}H_F^{(k)}(r,\xi)$ exists and 
is given by the right hand side of \eqref{limit}.

\begin{lem} \label{polarcoor}For any $L,\xi>0$ and $k\in \mathbb{N}$, we have
\begin{equation*}
\int_{D_k(L)}ds^{(k)}\myprod{i}{1}{k}\frac{F\left(\xi\,\sqrt{s_i-s_{i-1}}\right)}{\sqrt{s_i-s_{i-1}}}=2^k\int_{0}^{\sqrt{L}}dr\,\,r^{k-1}H_F^{(k)}(r,\xi).
\end{equation*}
\end{lem}
\begin{proof}
The change of variables $u_i=\sqrt{s_i-s_{i-1}}$ transforms $D_{k}(L)$ into
$$\set{(u_1,u_2,...,u_k)\in \R^k: u_i> 0, \mysum{i}{1}{k}u_i^2< L},$$ with Jacobian satisfying
$$\frac{\partial(s_1,...,s_k)}{\partial(u_1,..,u_k)}=2^k\myprod{i}{1}{k}u_i.$$
Consequently, changing to polar coordinates yields the desired result.
\end{proof}

\section{Proof of Theorem \ref{mainthm}}\label{sec:props}
The proof consists of several steps. We begin with the observation that  the  Fourier inversion formula can be applied to any function  $f\in\mathcal{D}$ to conclude that 
\begin{equation}
\int_{0}^{t^2a^2(n)}ds\,f(X_s)=(2\pi)^{-1}\int_{0}^{t^2a^2(n)}ds
\int_{\R}dx\,\wh{f}(x)\,e^{\dot{\iota}x\cdot X_s}.
\end{equation}

\begin{prop}\label{P1}
Let $\delta>0$ and set
\begin{equation*}
F_{n,1}(\delta,t)=\frac{1}{a(n)}\int_{0}^{t^2a^2(n)}ds\int_{|x|>\delta}dx\,\wh{f}(x)\,e^{\dot{\iota}x\cdot X_s}.
\end{equation*}
Then, $\tgo{F_{n,1}(\delta,t)}{0}$ in $L^2$, as $\tgo{n}{\infty}.$ 
\end{prop}
\begin{proof}
Observe that  \eqref{symmetric} and \eqref{incmts} with $k=2$ gives, after a suitable change of variables, 
\begin{align*}
&\E{\abs{F_{n,1}(\delta,t)}^2}
=\E{F_{n,1}(\delta,t)\overline{F_{n,1}(\delta,t)}}\\
&=\frac{2}{a^2(n)}\int_{0}^{t^2a^2(n)}ds_1\int_{s_1}^{t^2a^2(n)}ds_2
\int_{\substack{\,\,|x_2|>\delta,\\ |x_1|>\delta}}\wh{f}(x_2)\wh{f}(x_1)e^{-
\Psi(x_2)(s_2-s_1)-\Psi(x_2+x_1)s_1}dx^{(2)}\\ \nonumber
&=\frac{2}{a^2(n)}\int_{0}^{t^2a^2(n)}ds_1\int_{s_1}^{t^2a^2(n)}ds_2
\int_{\substack{\,\,\,\,\,\,\,\,\,\,\,|y_2|>\delta,\\  |y_1-y_2|>\delta}}\wh{f}(y_2)\wh{f}(y_1-y_2)e^{-
\Psi(y_2)(s_2-s_1)-\Psi(y_1)s_1}dy^{(2)}.
\end{align*}
Since $\Psi(x)$ is non-decreasing on $[0,\infty)$ and radial,  
for $|y_2|>\delta$ we have that
\begin{equation*}
\int_{s_1}^{t^2a^2(n)}ds_2\,e^{-\Psi(y_2)s_2}\leq e^{-\Psi(y_2)s_1}\, \Psi(y_2)^{-1}
\leq e^{-\Psi(y_2)s_1}\,\Psi(\delta)^{-1}.
\end{equation*}
The latter inequality, $\wh{f}\in L^1(\R)$ and \eqref{invdens} give that

\begin{align*}\label{ineexp1}
&\E{\abs{F_{n,1}(\delta,t)}^2}\\ &\leq\frac{2}{a^2(n)\Psi(\delta)}\int_{0}^{t^2a^2(n)}ds_1
\int_{|y_2|>\delta}dy_2\abs{\wh{f}(y_2)}\int_{|y_1-y_2|>\delta}dy_1\abs{\wh{f}(y_1-y_2)}e^{-\Psi(y_1)s_1}\nonumber \\ 
&\leq\frac{2||\wh{f}||_{\infty}}{a^2(n)\Psi(\delta)}\int_{0}^{t^2a^2(n)}ds_1
\int_{|y_2|>\delta}dy_2\abs{\wh{f}(y_2)}\int_{\R}dy_1\,e^{-\Psi(y_1)s_1}\\\nonumber
&\leq \frac{4\pi||\wh{f}||_{\infty}||\wh{f}||_{1}}{\Psi(\delta)}\left(\frac{1}{a^2(n)}\int_{0}^{t^2a^2(n)}ds_1\,p_{s_1}(0)\right).
\end{align*}

By  \eqref{asympdensity2}, the last term in the above inequality
vanishes as $\tgo{n}{\infty}$ and this completes the proof.
\end{proof}

\begin{prop}
Consider $\delta>0$ and set
\begin{equation*}
F_{n,2}(\delta,t)=\frac{1}{a(n)}\int_{0}^{t^2a^2(n)}ds\int_{|x|\leq\delta}dx\left(\wh{f}(x)-\wh{f}(0)\right)e^{\dot{\iota}x \cdot X_s}.
\end{equation*}
Then, $\tgo{\E{\abs{F_{n,2}(\delta,t)}^2}}{0}$, as $\tgo{n}{\infty}.$
\end{prop}
\begin{proof}
Recall that by \eqref{ine1}, $\Psi(y_2)\geq y_2^{2}\,\underline{\ell}(\delta)$ for any $|y_2|\leq\delta.$  
Next, by appealing to \eqref{deltacond},  we have by mimicking the proof of the latter  proposition that
\begin{align*}
&\E{\abs{F_{n,2}(\delta,t)}^2}\\ &\leq \frac{2}{a^2(n)}
\int_{0}^{t^2a^2(n)}ds_1
\int_{|y_2|\leq\delta}\int_{|y_1-y_2|\leq\delta}\frac{|\wh{f}(y_2)-\wh{f}(0)|}{\Psi(y_2)}
|\wh{f}(y_1-y_2)-\wh{f}(0)|e^{-\Psi(y_1)s_1}dy^{(2)} \\ \nonumber
&\leq \frac{2}{a^2(n)}
\int_{0}^{t^2a^2(n)}ds_1
\int_{|y_2|\leq\delta}\int_{|y_1-y_2|\leq\delta}\frac{|\wh{f}(y_2)-\wh{f}(0)|}{\underline{\ell}(\delta)\,y_2^2}
|\wh{f}(y_1-y_2)-\wh{f}(0)|e^{-\Psi(y_1)s_1}dy^{(2)} \\ \nonumber
&\leq \frac{8\, \pi \,||\wh{f}||_{\infty}}{\underline{\ell}(\delta)}\left(\int_{\R}dy_2\,\frac{|\wh{f}(y_2)-\wh{f}(0)|}{y_2^2}\right) \left(\frac{1}{a^2(n)}
\int_{0}^{t^2a^2(n)}ds_1\,p_{s_1}(0)\right),
\end{align*} 
where the last term converges to $0$, according to \eqref{asympdensity2}.
\end{proof}

In order to simplify the proof of the next Proposition, we establish the following lemma.
\begin{lem}\label{basiclem}
Let $H:[0,\infty)\rightarrow [0,\infty)$ be an increasing function.
Then, for every  $k>0$ and $0<\varepsilon<1$, the following inequalities hold
\begin{align}\label{Hineq}
\frac{(1-\varepsilon^k)}{k}H(\varepsilon L)\leq \frac{1}{L^k}\int_{0}^{L}\, dr\, r^{k-1}H(r)\leq \frac{1}{k}
\lim\limits_{\tgo{r}{\infty}}H(r), 
\end{align}
for any $L>0.$ In particular,
\begin{align}\label{Hlim}
\lim\limits_{ \tgo{L}{\infty}}\frac{1}{L^k}\int_{0}^{L}\, dr\, r^{k-1}\,H(r)=\frac{1}{k}
\lim\limits_{\tgo{r}{\infty}}H(r).
\end{align}
\end{lem}
\begin{proof}
By assumption, we have $0\leq H(\varepsilon L) \leq H(r),\,\, \, \varepsilon L<r,$ so that  
\begin{align*} 
 \frac{(1-\varepsilon^k)}{k}H(\varepsilon L)\leq
 \frac{1}{L^k}\int_{\varepsilon L}^{L}\, dr\, r^{k-1}H(r)\leq
 \frac{1}{L^k}\int_{0}^{L}\, dr\, r^{k-1}H(r).
\end{align*}
A similar argument provides the other desired inequality in \eqref{Hineq}. Therefore, using that $k>0$, we conclude \eqref{Hlim} by letting first $L$ go to  $\infty$ and then letting $\varepsilon$ go to zero in \eqref{Hineq}.
\end{proof} 

\begin{prop}\label{intpropo}
Let  $\delta>0$ and consider 
\begin{equation}\label{Fn}
F_n(\delta,t)= \frac{1}{a(n)}\int_{0}^{t^2a^2(n)}ds\int_{|x|\leq\delta}dx\,e^{\dot{\iota}x\cdot X_s}=\frac{2}{a(n)}\int_{0}^{t^2a^2(n)}
\frac{\sin{\left(\delta X_s\right)}}{X_s}ds.
\end{equation}
Then, for all $k \in \mathbb{N}$, we have 
\begin{align}\label{finmom1}
\sup\limits_{n\in \mathbb{N}}\E{F_n^{2k}(\delta,t)}&\leq
 \pthesis{\frac{\,\pi^{2} \,t^{2}}{\underline{\ell}(2k\delta)}}^k
 \frac{(2k)!}{k!}. 
\end{align}
Thus the sequence  $\set{F_n^{k}(\delta,t): n\in\mathbb{N}}$
is uniformly integrable for every $\delta$. Moreover,
\begin{align}\label{finmom}
\lim \limits_{ \tgo{\delta}{0^+}} \varliminf\limits_{ 
\tgo{n}{\infty}}\E{ F_{n}^k(\delta,t)}&=
\lim\limits_{ \tgo{\delta}{0^+} }\varlimsup\limits_{\tgo{n}{\infty}}\E{F_{n}^k(\delta,t)}\\ \nonumber
&=\pthesis{\frac{\pi \,t}{\sqrt{\ell}}}^{k}\frac{k!}{\Gamma
\pthesis{1+\frac{k}{2}}},
\end{align}
where  the constant $\ell$ and the function $\underline{\ell}$ are as defined in \eqref{ellandphi} and \eqref{ine1}, respectively.
\end{prop}
\begin{proof}
Under the notation given in the previous section (and due to \eqref{incmts}) we have that
$\E{F_n^k(\delta,t)}$ is equal to
\begin{align*}
&\,\,\,\,\,\,\,\, \frac{k!}{a^k(n)}\int_{D_k(t^2a^2(n))}ds^{(k)}\int_{\set{|x_i|\leq\delta, \,i=1,..,k}}dx^{(k)}\,e^{-\mysum{i}{1}{k}\Psi(\mysum{j}{i}{k}x_j)(s_i-s_{i-1})} \\ \nonumber
&=\frac{k!}{a^k(n)}\int_{D_k(t^2a^2(n))}ds^{(k)}\int_{\set{|y_k|\leq\delta,\,
|y_i-y_{i+1}|\leq \delta,\,i=1,..,k-1}}dy^{(k)}\,e^{-\mysum{i}{1}{k}\Psi(y_i)(s_i-s_{i-1})},
\end{align*}
where  the last equality is obtained by making the change of variables $y_i=\mysum{j}{i}{k}x_j$.

We also note that 
\begin{align}\label{subsets}
\set{|y_i|\leq \delta/2,i=1,...,k}&\subset{\set{|y_k|\leq\delta,
|y_i-y_{i+1}|\leq \delta,i=1,...,k-1}}\nonumber\\
&\subset\set{|y_i|\leq k\delta,i=1,...,k}.
\end{align}

Before proceeding,  let us define
\begin{equation}\label{Ndist}
 F(\lambda)=\frac{1}{\sqrt{2\pi}}\int_{|y|\leq \lambda}dy\, e^{-y^2/2},
\end{equation}
 so that for every $A,M>0$, we have
\begin{equation*}
\int_{\set{|y_i|\leq M\delta, i=1,..,k}}dy^{(k)}\,e^{-\mysum{i}{1}{k}y_i^2(s_i-s_{i-1})A}=\left(\frac{\pi}{A}\right)^{k/2}\myprod{i}{1}{k}\frac{F(M\delta\sqrt{2A(s_i-s_{i-1})})}{\sqrt{s_i-s_{i-1}}}.
\end{equation*}

Let $k\in \mathbb{N}$. By combining  \eqref{ine1}, \eqref{subsets} and Lemma \ref{polarcoor}, we obtain that $\E{F_n^k(\delta,t)}$ is bounded above and below  by terms of the form
\begin{align}\label{generalbound}
\left(\frac{\pi}{A}\right)^{k/2}\frac{2^k\,k!}{a^k(n)}
\int_{0}^{ta(n)}dr\,\,r^{k-1}H_F^{(k)}\pthesis{r,M\delta\sqrt{2A}}
\end{align}
where
\begin{enumerate}
\item [$(i)$] for the lower--bound, $A=\overline{\ell}(2^{-1}\delta)$  and  $M=2^{-1}$,
\item[$(ii)$] for the upper--bound, $A=\underline{\ell}(\delta\, k)$ and  $M=k$.
\end{enumerate}
Next, by  putting together  \eqref{limit} with \eqref{Hlim},  and
recalling that   $|\mathbb{S}^{k-1}|=2\,\pi^{k/2}\,\Gamma^{-1}(k/2)$, we deduce
\begin{align}\label{lim4}
\lim\limits_{\tgo{n}{\infty}}\frac{1}{a^k(n)}\int_{0}^{ta(n)}dr
\,\,r^{k-1}H_F^{(k)}(r,M\delta\sqrt{2A})&=
\frac{t^k}{k}\lim\limits_{\tgo{r}{\infty}}
H_F^{(k)}(r,M\delta\sqrt{2A}) \\ \nonumber&=
\pthesis{\frac{\pi^{1/2}\,t}{2}}^k\frac{1}{\Gamma\pthesis{1+\frac{k}{2}}},
\end{align}
for any $\delta,M,A>0$ fixed and all $k\geq1$. From the above limit, we conclude by \eqref{generalbound} and  \eqref{Hineq} that
\begin{equation*}
\E{F_n^k(\delta,t)}\leq \left(\frac{t\,\pi}{\sqrt{\underline{\ell}(k\, \delta)}}\right)^{k}\, \frac{k!} {\Gamma\pthesis{1+\frac{k}{2}}}, 
\end{equation*}
which  implies  \eqref{finmom1}. On the other hand, the fact that $\E{F_n^k(\delta,t)}$ is bounded by terms of the form \eqref{generalbound} and the foregoing  limit
\eqref{lim4} yield
\begin{equation*}
\left(\frac{t\,\pi}{\sqrt{\overline{\ell}(2^{-1}\, \delta)}}\right)^{k}\, \frac{k!} {\Gamma\pthesis{1+\frac{k}{2}}}\leq \varliminf\limits_{ 
\tgo{n}{\infty}}\E{F_{n}^k(\delta,t)}
\end{equation*}
and
\begin{equation*}
\varlimsup\limits_{ 
\tgo{n}{\infty}}\E{F_{n}^k(\delta,t)}\leq \left(\frac{t\,\pi}{\sqrt{\underline{\ell}(k\, \delta)}}\right)^{k}\, \frac{k!} {\Gamma\pthesis{1+\frac{k}{2}}}. 
\end{equation*}
Thus, \eqref{finmom} follows by letting $\tgo{\delta}{0^{+}}$ in the above expressions and using \eqref{limitc}.
\end{proof}

We now observe that the proof of the Theorem \ref{mainthm} follows by setting 
\begin{align*}\label{Idef}
I_{n,\delta}^{(1)}(t)&=(2\pi)^{-1}\left(F_{n,1}(\delta,t)+F_{n,2}(\delta,t)
\right),\nonumber \\
I_{n,\delta}^{(2)}(t)&=(2\pi)^{-1}F_{n}(\delta,t).
\end{align*}

\begin{rmk} We point out that Proposition \ref{intpropo} remains true for any symmetric L\'{e}vy process with characteristic exponent satisfying only condition \eqref{ellandphi}.
\end{rmk}

\section{proof of  theorem  \ref{weaklim} and corollary \ref{maincor} }\label{sec:maincor}
{\bf Proof of Theorem \ref{weaklim}:}
We will begin by  showing that the sequence 
$$\set{I_{n,\delta}^{(2)}(t)=(2\pi)^{-1}F_n(\delta,t): n\in \mathbb{N}},$$
with $F_n(\delta,t)$ as defined in Proposition \ref{intpropo} is Cauchy in
$L^{2}$ under the assumption \eqref{seqcond}.
 Let $N,n\in \mathbb{N}$. It is easy to see that
\begin{align*}
F_{n+N}(\delta,t)-F_{n}(\delta,t)=
\mathcal{I}_{\,n,N}(\delta,t)+ \mathcal{II}_{\,n,N}(\delta,t),
\end{align*}
where
\begin{align*}
\mathcal{I}_{\,n,N}(\delta,t)&=\frac{1}{a(n+N)}\int_{t^2a^2(n)}^{
t^{2}a^2(n+N)}ds\,\int_{\abs{x}\leq \delta}dx\,e^{\dot{\iota}x\cdot X_s },\\ 
\mathcal{II}_{\,n,N}(\delta,t)&=\pthesis{\frac{a(n)}{a(n+N)}-1}F_{n}(
\delta,t).
\end{align*}

We observe by appealing to \eqref{finmom1} with $k=1$ and
\eqref{seqcond}  that
$$\lim\limits_{\tgo{n}{\infty}}
\E{\abs{\mathcal{II}_{n,N}(\delta,t)}^2}\leq 
\lim\limits_{\tgo{n}{\infty}}\pthesis{\frac{a(n)}{a(n+N)}-1}^2\,
\frac{2\pi^2\,t^2}{\underline{\ell}(2\,\delta)}=0.$$
On the other hand, by employing the techniques developed for the
proof of the Proposition \ref{intpropo}, we have that $\E{\abs{\mathcal{I}_{n,N}(\delta,t)}^2}$ is bounded above by
\begin{align*}
&\frac{2\pi}{a^2(n+N)\,\underline{\ell}(2\delta)}\int_{t^2a^2(n)}^{t^2a^2(n+N)}ds_1
\int_{s_1}^{t^2a^2(n+N)}ds_2
\myprod{i}{1}{2}\frac{F\pthesis{\xi \, \sqrt{s_i-s_{i-1}}}}{\sqrt{s_i-s_{i-1}}}=\\
&\frac{8\pi}{a^2(n+N)\,\underline{\ell}(2\delta)}
\int_{ta(n)}^{ta(n+N)}du_1\int_{0}^{\sqrt{t^{2}a^2(n+N)-u_1^2}}du_2 \myprod{i}{1}{2}F \pthesis{\xi \,u_i},
\end{align*}
where $\xi=2\delta \sqrt{2\underline{\ell}(2\delta)}$ and the function $F$ is as defined in \eqref{Ndist}.
Thus, by making use of   the facts that $||F||_{\infty}= 1$ and
$$\int_{ta(n)}^{ta(n+N)}du_1\int_{0}^{\sqrt{t^{2}a^2(n+N)-u_1^2}}du_2=t^{2}a^2(n+N)\int_{\frac{a(n)}{a(n+N)}}^{1}
\sqrt{1-w^2},$$ we arrive by \eqref{seqcond} at
\begin{align*}
\lim\limits_{\tgo{n}{\infty}}\E{\abs{\mathcal{I}_{\,n,N}(\delta,t)}^2}\leq \lim\limits_{\tgo{n}{\infty}}
\frac{8\,\pi\, t^2}{\underline{\ell}(2\delta)} \int_{\frac{a(n)}{a(n+N)}}^{1}
\sqrt{1-w^2}=0.
\end{align*}
Therefore, by denoting $I_{\delta}(t)$ the limit in $L^{2}$ of the sequence $\set{I_{n,\delta}^{(2)}(t):n\in \mathbbm{N}}$ and 
 based on the classical results about weak convergence and uniformly integrability presented in \cite[\S25]{Bil}, we conclude that 
\begin{enumerate}
\item[$(a)$] 
$
\pthesis{I_{n,\delta}^{(2)}(t)}^{k} \overset{\mathcal{L}}{\longrightarrow}\,I_{\delta}^k(t),\,\, \tgo{n}{\infty}.
$
\item[$(b)$]
$I_{\delta}(t) \in L^{k}$, for all $k>1$ integer
(due to \eqref{finmom2} ) and
$\lim\limits_{\tgo{n}{\infty}}
\E{ \pthesis{I_{n,\delta}^{(2)}(t)}^k}$ exists and it equals
$\E{I_{\delta}^k(t)}$.
\end{enumerate}
The aforementioned conclusions allow us to apply Slutsky's  Theorem to conclude that 
\begin{align*}
I_{n}(t) \overset{\mathcal{L}}{\longrightarrow}\wh{f}(0)\,I_{\delta}(t),
\end{align*}
as $\tgo{n}{\infty}$ for every $\delta>0$. The above  weak limit  implies  that the   random variables $\set{I_{\delta}(t):\delta>0}$ are identically distributed so that 
$\E{I_1^k(t)}=\E{I_{\delta}^k(t)}$ for all $\delta>0$. Thus, by \eqref{limits} and part $(b)$ above, we arrive at
$$\E{I_1^k(t)}=\lim\limits_{\tgo{\delta}{0+}}\E{I_{\delta}^k(t)}=
\pthesis{\frac{t}{2\sqrt{\ell}}}^k \frac{k!}{\Gamma
\pthesis{1+\frac{k}{2}}}.$$
Hence, the proof of Theorem \ref{weaklim} is complete by
taking $I(t)=I_1(t)$.\qed
\\

{\bf Proof of Corollary \ref{maincor}:} 
For $f\in \mathcal{D}$, consider  the  decomposition
provided by Theorem \ref{mainthm}. Namely, for $ \delta>0$
  \begin{equation}\label{decomp.rel}
I_{n}(t)=\frac{1}{a(n)}\int_{0}^{t^2a^2(n)}ds\,f(X_s)= I_{n,\delta}^{(1)}(t)+\wh{f}(0)I_{n,\delta}^{(2)}(t).
 \end{equation}

Let us denote
$$A_{n}(t)=\sqrt{\E{\abs{I_n(t)}^2}}=||I_{n}(t)||_{L^{2}},$$
so that due to \eqref{decomp.rel}, we obtain
$$A_{n}(t)=|| I_{n,\delta}^{(1)}(t)+\wh{f}(0)I_{n,\delta}^{(2)}(t)||_{L^{2}}.$$

By employing the triangle inequality, we have
\begin{align*}
 -|| I_{n,\delta}^{(1)}(t)||_{L^{2}} +
 || \wh{f}(0)\,I_{n,\delta}^{(2)}(t)||_{L^{2}}
 \leq A_{n}(t)  
 \leq  || I_{n,\delta}^{(1)}(t)||_{L^{2}} +
 || \wh{f}(0)\,I_{n,\delta}^{(2)}(t)||_{L^{2}}.
\end{align*}
By part $(i)$ of Theorem \ref{mainthm}, we know that $I^{(1)}_{n,\delta}(t)$ converges to zero in $L^2$ as $\tgo{n}{\infty}$ (which also holds in $L^{1}$ again because of the Cauchy-Schwarz inequality) so that the last inequality implies

\begin{align*}
\abs{\wh{f}(0)}\varliminf \limits_{\tgo{n}{\infty}}||\,I_{n,\delta}^{(2)}(t)||_{L^{2}}
&\leq
\varliminf \limits_{\tgo{n}{\infty}}A_n(t)\\ &\leq 
\varlimsup \limits_{\tgo{n}{\infty}}A_n(t)\leq
\abs{\wh{f}(0)}
\varlimsup \limits_{\tgo{n}{\infty}}||\,I_{n,\delta}^{(2)}(t)||_{L^{2}}.
\end{align*}
Thus, by letting $\tgo{\delta}{0+}$ in the above expression  and appealing to part $(ii)$ of Theorem \ref{mainthm}, we deduce
$$\varliminf \limits_{\tgo{n}{\infty}}A_n^2(t)=\varlimsup \limits_{\tgo{n}{\infty}}A_n^2(t)=(\wh{f}(0))^2\,\frac{t^2}{2\ell},$$
which shows \eqref{seclim}.

Next, \eqref{firstlim} is easy to calculate by using part of the arguments given above and is left to the reader.  This completes the proof. \qed
\\

Finally, as mentioned in the introduction, the results in this paper were motivated by the Nualart--Xu results \cite{Nualart}.  It is interesting to note that many of the computations for the Fourier transform in \cite{Nualart} are similar to those used by Ba\~nuelos and S\'a Barreto \cite{Ba.Sab} and in the author's paper \cite{Acuna} to compute the heat invariants for Schr\"odinger operators for the Laplacian and the fractional Laplacian. In these papers  one uses Fourier transform methods to obtain estimates on  Feynman--Kac  expressions of the form 
\begin{align}\label{expfunct}
\mathbb{E}^{t}_{x,x}\left[e^{-\int_{0}^{t}V(X_s)ds}\right].
\end{align}
Here, $X$ is the symmetric $\alpha$-stable process and  $\mathbb{E}^{t}_{x,x}$ stands for the expectation with respect to the process (stable bridge) starting at $x$ and conditioned to be at
$x$ at time $t$.  The function (potential $V$) is infinitely differentiable of compact support. One interesting problem is  to obtain estimates and properties of \eqref{expfunct} with less regularity on the functions $V$. In this direction, in \cite{Ba.Sel}, Ba\~{n}uelos and Selma employed the Taylor expansion of the exponential 
function  and probabilistic techniques  to investigate 
the $k$-th moment of $\int_{0}^{t}V(X_s)ds$ with respect to the stable bridge, for $V's$ which are H\"{o}lder continuous. 
Expressions similar to those in Nualar--Xu \cite{Nualart} for computation of moments are derived.

{\bf Acknowledgements}: I am grateful  to my supervisor,  Professor Rodrigo Ba\~nuelos, for his valuable suggestions and time while preparing this paper.   I wish to thank the referee for the many comments and suggestions which considerably improved the paper.

\end{document}